\documentclass[letterpaper, twoside, noamsfonts, 11 pt]{amsart}

\usepackage{michael}
\usepackage{textcomp}

\usetikzlibrary{decorations, calc, intersections, through, positioning}
\usetikzlibrary{decorations.markings}
\usetikzlibrary{pgfplots.fillbetween}

\usepackage{float}

\numberwithin{equation}{section}

\begin{document}
  \title{Maximal Operators and Fourier Restriction on the Moment Curve}
  \author{Michael Jesurum}
  \maketitle
  \begin{abstract}
    We bound certain $r$-maximal restriction operators on the moment curve.
  \end{abstract}
  \section{Introduction} \label{introduction}
  %
    % This holds the introduction to Maximal Operators and Fourier Restriction on
% the Moment Curve.
%
%
Let $\gamma(t) = (t, \frac{1}{2} t^{2}, \dots, \frac{1}{d} t^{d})$, and let
$\Gamma$ be the image of this curve for $t \in \mathbb{R}$. Drury~\cite{Drury}
proved the Fourier restriction estimate
\begin{equation*}
  ||\hat{f}||_{L^{q}(\Gamma)} \leq C_{p} ||f||_{L^{p}(\mathbb{R}^{d})}
\end{equation*}
for $1 \leq p < \frac{d^{2} + d + 2}{d^{2} + d}$ and
$p' = \frac{d (d + 1)}{2} q$. In the spirit of \cite{Ramos2}, we study the
$r$-maximal form of this restriction operator: $M_{r}\hat{f}|_{\Gamma}$, where
\begin{equation*}
  M_{r}h(x) = \bigg(\sup_{s > 0}\  \dashint_{B(x, s)} |h|^{r}\bigg)^{1/r}
\end{equation*}
and $1 \leq r < \infty$. For $d \geq 3$, we have the following maximal
restriction theorem:
\begin{theorem} \label{main theorem}
  For $q = \frac{2}{d (d + 1)} p' > p$ and $r$ satisfying
  %
  % r in terms of p
  %
  % \begin{equation*}
  % %
  %   \begin{cases}
  %   %
  %     1 \leq p < \frac{d^{2} + d + 2}{d^{2} + d}
  %     \quad& 1 \leq r \leq 2
  %     \\
  %     1 \leq p < \frac{d^{2} + d + 2 (r - 1)}{d^{2} + d + 2 (r - 2)}
  %     \quad& 2 < r < \frac{d + 2}{2}
  %     \\
  %     1 \leq p \leq \frac{r d}{r d - 1}
  %     \quad& \frac{d + 2}{2} \leq r
  %   %
  %   \end{cases}
  % %
  % \end{equation*}
  %
  % p in terms of r
  %
  \begin{equation*}
    \begin{cases}
      r \leq \frac{p'}{d},
      &\text{if} \quad 1 \leq p \leq \frac{d^{2} + 2 d}{d^{2} + 2 d - 2};
      \\
      r < p' - \frac{d^{2} + d - 2}{2},
      &\text{if} \quad
      \frac{d^{2} + 2 d}{d^{2} + 2 d - 2} < p < \frac{d^{2} + d + 2}{d^{2} + d},
    \end{cases}
  \end{equation*}
  we have the following estimate for every $f \in L^{p}(\mathbb{R}^{d})$:
  \begin{equation} \label{main theorem estimate}
    ||M_{r}\hat{f}||_{L^{q}(\Gamma)}
    \leq C_{p, r} ||f||_{L^{p}(\mathbb{R}^{d})}.
  \end{equation}
\end{theorem}
M\"uller, Ricci, and Wright~\cite{MullerRicciWright} introduced maximal
restriction theorems to obtain a pointwise interpretation of the
restriction operator associated to $C^{2}$ curves in \(\mathbb{R}^{2}\).
After proving bounds for a two-parameter maximal restriction operator, they
introduced the operator
\begin{equation} \label{M2}
  M_{2}h(x) = (M|h|^{2})^{1/2}(x)
\end{equation}
to aid with bounds for a strong maximal restriction operator.
Following their logic, for the restriction operator \(\mathcal{R}\) associated 
to the moment curve \(\gamma\), the case \(r = 2\) in
Theorem~\ref{main theorem} implies:
\begin{corollary} \label{pointwise corollary}
  Let \(f \in L^{p}(\mathbb{R}^{d})\) and
  \(1 \leq p < \frac{d^{2} + d + 2}{d^{2} + d}\). With respect to arclength
  measure, almost every \(x \in \Gamma\) is a Lebesgue point for \(\hat{f}\)
  and the regularized value of \(\hat{f}\) at \(x\) coincides with
  \(\mathcal{R}f(x)\).
  %
  % This includes both points in MRW, but I think I only want the second.
  % \begin{enumerate} [label = (\roman*)]
  % %
  %   \item
  %   %
  %     Let \(\chi \in \mathcal{S}(\mathbb{R}^{d})\) with \(\int \chi = 1\), and
  %     for \(\epsilon > 0\) let
  %     %
  %     \begin{equation*}
  %     %
  %       \chi_{\epsilon}(x) = \epsilon^{-d}\chi\Big(\frac{x}{\epsilon}\Big).
  %     %
  %     \end{equation*}
  %     %
  %     Then, with respect to arclength measure, for almost every
  %     \(x \in \Gamma\),
  %     %
  %     \begin{equation*}
  %     %
  %       \lim_{\epsilon \rightarrow 0} \hat{f} \ast \chi_{\epsilon}(x)
  %       = \mathcal{R}f(x).
  %     %
  %     \end{equation*}
  %   %
  %   \item
  %   %
  %     With respect to arclength measure, almost every \(x \in \Gamma\) is a
  %     Lebesgue point for \(\hat{f}\) and the regularized value of \(\hat{f}\)
  %     at \(x\) coincides with \(\mathcal{R}f(x)\).
  % %
  % \end{enumerate}
%
\end{corollary}
\newpage
Later, Vitturi~\cite{Vitturi} proved similar maximal restriction estimates in
the case of the unit sphere in any dimension $d \geq 3$. Ramos~\cite{Ramos1}
improved the known results on spheres in all dimensions, and then
in~\cite{Ramos2} focused on dimensions $d = 2$ and $d = 3$. In particular, he
generalized the operator~\eqref{M2}~to
\begin{equation*}
  M_{r}h = (M|h|^{r})^{1/r}(x)
\end{equation*}
for $1 \leq r < \infty$, and Theorem~2 in that paper was a maximal restriction
result for this operator on the unit circle for $p < \frac{4}{3}$ and
$r \leq 2$. Thus, in the case $d = 2$, Theorem~\ref{main theorem} is due to
Ramos~\cite{Ramos2}, since the arguments that apply to the circle also
apply to the~parabola.
\par
Kova\v c~\cite{Kovac} took a more general approach, proving maximal and
variational restriction estimates using restriction inequalities as a black
box. Theorem~1, Remark~2, and Remark~3 in that paper combine with
Drury's~\cite{Drury} restriction estimate to show that
\begin{equation*}
  ||M_{2}\hat{f}||_{L^{q}(\Gamma)} \leq C_{p} ||f||_{L^{p}(\mathbb{R}^{d})}
\end{equation*}
for $1 \leq p < \frac{d^{2} + d + 2}{d^{2} + d + 1}$ and
$p' = \frac{d (d + 1)}{2} q$. Theorem~\ref{main theorem} extends this range
of $p$ to the full Drury range for $M_{2}$ and gives estimates of the
form~\eqref{main theorem estimate} for $r > 2$. See
also~\cite{KovacSilva} for more on variational restriction theorems.
\par
In the case $d \geq 3$, the first two cases of Theorem~\ref{main theorem} are
distinct. When $r = 2$, we obtain the full range given by Drury:
$p < \frac{d^{2} + d + 2}{d^{2} + d}$. For $r = \frac{d + 2}{2}$, the range
of $p$ corresponds to the Christ-Prestini
$p < \frac{d^{2} + 2 d}{d^{2} + 2 d - 2}$ (see~\cite{Christ}
and~\cite{Prestini}). Figure~\ref{fig: range of r and p} illustrates these
ranges.
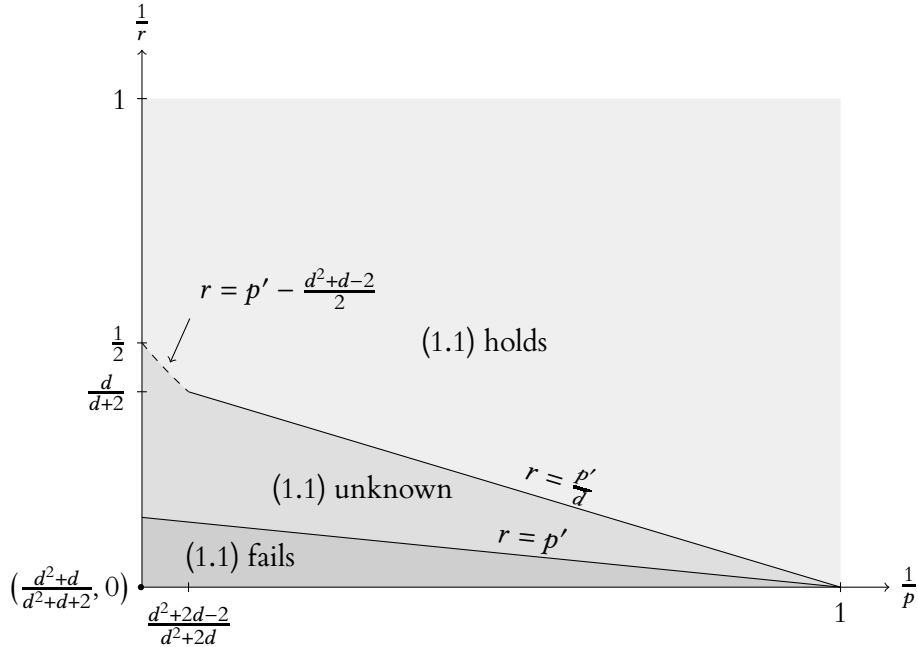
\begin{figure}[H]
  \centering
  \begin{tikzpicture}[xscale = 65, yscale = 6.5]
  %
    % main theorem estimate is good.
    %
    \fill [fill = lightgray!25!white]
    ({6 / 7}, 1)
    -- ({6 / 7}, .5)
    [domain = {6 / 7} : {13 / 15}]
    plot (\x, {(1 - \x) / (5*\x - 4)})
    -- (1, 0)
    -- (1, 1)
    -- ({6 / 7}, 1);
    %
    % main theorem estimate is a maybe.
    %
    \fill [fill = lightgray!50!white]
    ({6 / 7}, .5)
    [domain = {6 / 7} : {13 / 15}]
    plot (\x, {(1 - \x) / (5*\x - 4)})
    -- (1, 0)
    -- ({6 / 7}, {1 / 7})
    -- ({6 / 7}, .5);
    %
    % main theorem estimate is bad.
    %
    \fill [fill = lightgray!75!white]
    ({6 / 7}, {1 / 7}) -- (1, 0) -- ({6 / 7}, 0) -- cycle;
    %
    % x axis
    %
    \draw[->] ({6 / 7}, 0) -- (1.01, 0) node [right] {$\frac{1}{p}$};
    %
    % y axis
    % 
    \draw[->] ({6 / 7}, 0) -- ({6 / 7}, 1.1) node [above] {$\frac{1}{r}$};
    %
    % origin
    %
    \coordinate
    [label = left: {$\big(\frac{d^{2} + d}{d^{2} + d + 2}, 0\big)$}]
    (O) at ({6 / 7}, 0);
    %
    % Label for origin. The radii are off by a factor of 10 because of the
    % difference in scale factors for x and y.
    %
    \filldraw (O) circle [x radius = .015 pt, y radius = .15 pt];
    %
    % Labels of 1 on each axis.
    %
    \draw (1, -.01) node [below] {1} -- (1, .01);
    \draw ({6 / 7 -.001}, 1) node [left] {1} -- ({6 / 7 + .001}, 1);
    %
    % Main range of r from the theorem.
    %
    \draw [dashed, domain = {6 / 7} : {13 / 15}]
    plot (\x, {(1 - \x) / (5*\x - 4)});
    %
    % Labels for the previous line.
    %
    \draw ({6 / 7 - .001}, .5) node [left] {$\frac{1}{2}$}
    -- ({6 / 7 + .001}, .5);
    \coordinate [label = above right: {$r = p' - \frac{d^{2} + d - 2}{2}$}]
    (theoremLine) at ({6 / 7 + .01}, .55);
    \draw [->] (theoremLine)
    -- ($({(6 / 7 + 13 / 15) / 2}, .45) + (.001, .01)$);
    %
    % Range of r obtained from interpolation between Christ and infinity.
    % The atan(-.3) is because the slope of the line is -3, but x is scaled
    % by a factor of 10 compared with y.
    %
    \draw ({13 / 15}, .4)
    -- node [above right = - 3 pt, rotate = {atan(-.3)}]
    {$r = \frac{p'}{d}$} (1, 0);
    %
    % y coordinate label for the previous line.
    %
    \draw ({6 / 7 - .001}, .4) node [left] {$\frac{d}{d + 2}$}
    -- ({6 / 7 + .001}, .4);
    %
    % x coordinate label for the previous line.
    %
    \draw ({13 / 15}, -.01)
    node [below] {$\frac{d^{2} + 2 d - 2}{d^{2} + 2 d}$}
    -- ({13 / 15}, .01);
    %
    % Possible best range of r. Again, atan(-.1) is from the slope of -1
    % divided by 10.
    %
    \draw ({6 / 7}, {1 / 7})
    -- node [above right = -3 pt, rotate = {atan(-.1)}] {$r = p'$} (1, 0);
    \node at ({6 / 7 + .07}, .5) {\eqref{main theorem estimate} holds};
    \node at ({6 / 7 + .045}, .2) {\eqref{main theorem estimate} unknown};
    \node at ({6 / 7 + .02}, .06) {\eqref{main theorem estimate} fails};
  \end{tikzpicture}
  \caption{\small \textit{Range of $r$ and $p$ for which the $r$-maximal
  restriction operator is bounded from $L^{p}$ to $L^{q}$, where
  $q = \frac{2}{d (d + 1)} p'$.}}
  \label{fig: range of r and p}
\end{figure}
\subsection*{Outline of Proof}
  The overall structure follows Drury's
  induction scheme from \cite{Drury}. To accomodate this, we prove the
  superficially stronger result:
  \begin{proposition} \label{main theorem first reduction}
    Let $2 \leq r < \frac{d + 2}{2}$. Denote
    by $M_{r}^{k}$ the $k$-fold composition of $M_{r}$ with itself. Then for
    each $k \in \mathbb{N}$,
    $1 \leq p < \frac{d^{2} + d + 2 (r - 1)}{d^{2} + d + 2 (r - 2)}$, and
    $p' = \frac{1}{2} d (d + 1) q$, we have
    \begin{equation*}
      ||M_{r}^{k}\hat{f}||_{L^{q}(\Gamma)}
      \leq C_{p, r} ||f||_{L^{p}(\mathbb{R}^{d})}.
    \end{equation*}
  \end{proposition}
  Indeed, for $r \geq \frac{d + 2}{2}$, we can interpolate the above result
  with the bound
  \begin{equation*}
    ||M_{\infty}\hat{f}||_{L^{\infty}(\Gamma)}
    \leq ||f||_{L^{1}(\mathbb{R}^{d})},
  \end{equation*}
  which follows from Hausdorff-Young. For $1 \leq r < 2$, we can apply
  H\"older's inequality to see that
  \begin{equation*}
    M_{r}\hat{f}(x) \leq M_{2}\hat{f}(x).
  \end{equation*}
  Thus, Theorem~\ref{main theorem} follows from
  Proposition~\ref{main theorem first reduction}. To prove
  Proposition~\ref{main theorem first reduction}, we first linearize the
  operator $f \mapsto M_{r}^{k}\hat{f}|_{\Gamma}$ 
  (Section~\ref{linearization}). Then, in Section~\ref{proof}, we apply the
  induction hypothesis to prove a mixed-norm estimate for the $d$-fold power of
  the linear operator from Section~\ref{linearization}. We interpolate this
  estimate with an $L^{2}$ bound for that same operator that comes from
  Plancherel. This interpolation allows us to increase the value of $p$, which
  completes the induction.
\subsection*{Acknowledgements}
  The author would like to thank Betsy Stovall for suggesting this project and
  for advising throughout the process. He also thanks the anonymous referee
  for their comments and suggestions. This work was supported by NSF grant
  DMS-1653264.
%%% Local Variables:
%%% mode: latex
%%% TeX-master: "main"
%%% End:

  %
  \section{Kolmogorov-Seliverstov-Plessner Linearization} \label{linearization}
  %
    % This houses the section Kolmogorov-Saliverstov Linearization for the paper
% Maximal Operators and Fourier Restriction to the Moment Curve.
%
%
We first fix $2 \leq r < \frac{d + 2}{2}$. This value of $r$ will remain fixed
throughout this section and the next. The first step is to linearize the
maximal operator given in Proposition~\ref{main theorem first reduction}. The
technique here is similar to~\cite{Ramos2}, which built on the
techniques of \cite{MullerRicciWright}, \cite{Vitturi}, and \cite{Ramos1}.
\par
Let $\chi_{a}(x)$ be the $L^{1}$-normalized characteristic function of the ball
of radius $a$; that is,
\begin{equation*}
  \chi_{a}(x) = \frac{1}{|B_{a}|} \chi \bigg( \frac{x}{a} \bigg),
\end{equation*}
where $B_{a}$ is the ball centered at 0 with radius $a$. Let
$\rho_{1}, \ldots, \rho_{k} \colon \mathbb{R}^{d} \rightarrow \mathbb{R}_{> 0}$
be measurable and
$\eta \colon \mathbb{R}^{d} \times \mathbb{R}^{d} \rightarrow \mathbb{C}$
be a measurable function such that
\begin{equation} \label{eta assumption}
  \dashint_{B_{\rho_{1}(x)}(x)} \ldots 
  \dashint_{B_{\rho_{k}(y_{k - 1})}(y_{k - 1})} |\eta(x, y_{k})|^{r'}
  \mathrm{d} y_{k} \ldots \mathrm{d} y_{1}
  \leq 1
\end{equation}
for every $x \in \mathbb{R}^{d}$. Set
\begin{equation} \label{A operator}
  \mathcal{A}^{k}_{x}(z_{1}, \ldots, z_{k})
  = \eta(x, x - z_{1} - \ldots - z_{k})
  \chi_{\rho_{1}(x)}(z_{1}) 
  \chi_{\rho_{2}(x - z_{1})}(z_{2}) \ldots
  \chi_{\rho_{k}(x - z_{1} - \ldots - z_{k - 1})}(z_{k}).
\end{equation}
and
\begin{equation} \label{M operator A form}
  M_{r, \eta, \rho, k} f(x)
  = \int_{\mathbb{R}^{k d}} \hat{f}(x - z_{1} - \ldots - z_{k})
  \mathcal{A}^{k}_{x}(z_{1}, \ldots, z_{k})
  \mathrm{d} z_{1} \ldots \mathrm{d} z_{k},
\end{equation}
or, equivalently,
\begin{equation} \label{M operator average form}
  M_{r, \eta, \rho, k} f(x)
  = \dashint_{B_{\rho_{1}(x)}(x)} \ldots 
  \dashint_{B_{\rho_{k}(z_{k - 1})}(z_{k - 1})} \hat{f}(z_{k}) \eta(x, z_{k})
  \mathrm{d} z_{k} \ldots \mathrm{d} z_{1}.
\end{equation}
\begin{lemma}
  Suppose there is $C > 0$ such that for every $\eta_{x}$ and
  $\rho_{1}, \ldots, \rho_{k}$ as above, and for all $f$ Schwartz,
  \begin{equation*}
    ||M_{r, \eta, \rho, k}f||_{L^{q}(\Gamma)}
    \leq C ||f||_{L^{p}(\mathbb{R}^{d})}.
  \end{equation*}
  Then
  \begin{equation*}
    ||M_{r}^{k}\hat{f}||_{L^{q}(\Gamma)} \leq C ||f||_{L^{p}(\mathbb{R}^{d})}.
  \end{equation*}
  for all Schwartz functions $f$.
\end{lemma}
\begin{proof}
  Let $f$ be a Schwartz function, and let
    \begin{equation*}
      \eta(x, y) = \frac{\overline{\hat{f}(y)}| 
      \hat{f}(y)|^{r - 2}}{\bigg( \dashint_{B_{\rho_{1}(x)}(x)} \ldots 
      \dashint_{B_{\rho_{k}(z_{k - 1})}(z_{k - 1})} |f(z_{k})|^{r}
      \mathrm{d} z_{k} \ldots \mathrm{d} z_{1}\bigg)^{\frac{1}{r'}}}.
    \end{equation*}
  Then for any $x \in \mathbb{R}^{d}$,
  \begin{align*}
    &\dashint_{B_{\rho_{1}(x)}(x)} \ldots 
    \dashint_{B_{\rho_{k}(y_{k - 1})}(y_{k - 1})} |\eta(x, y_{k})|^{r'}
    \mathrm{d} y_{k} \ldots \mathrm{d} y_{1}
    \\
    &= \frac{\dashint_{B_{\rho_{1}(x)}(x)} \ldots 
    \dashint_{B_{\rho_{k}(y_{k - 1})}(y_{k - 1})} |f(y_{k})|^{r'(r - 1)}
    \mathrm{d} y_{k} \ldots \mathrm{d} y_{1}}
    {\dashint_{B_{\rho_{1}(x)}(x)} \ldots 
    \dashint_{B_{\rho_{k}(z_{k - 1})}(z_{k - 1})} |f(z_{k})|^{r}
    \mathrm{d} z_{k} \ldots \mathrm{d} z_{1}}.
  \end{align*}
  Since $r' (r - 1) = r$, the numerator and denominator are equal and hence
  $\eta$ satisfies~\eqref{eta assumption}. Moreover,
  using~\eqref{M operator average form}, we have
  \begin{equation*}
    M_{r, \eta, \rho, k} f(x)
    = \frac{\dashint_{B_{\rho_{1}(x)}(x)} \ldots 
    \dashint_{B_{\rho_{k}(z_{k - 1})}(z_{k - 1})} |\hat{f}(z_{k})|^{r}
    \mathrm{d} z_{k} \ldots \mathrm{d} z_{1}}
    {\bigg( \dashint_{B_{\rho_{1}(x)}(x)} \ldots 
    \dashint_{B_{\rho_{k}(z_{k - 1})}(z_{k - 1})} |f(z_{k})|^{r}
    \mathrm{d} z_{k} \ldots \mathrm{d} z_{1}\bigg)^{\frac{1}{r'}}}.
  \end{equation*}
  Thus, we obtain
  \begin{equation*}
    M_{r, \eta, \rho, k} f(x) = 
    \bigg( \dashint_{B_{\rho_{1}(x)}(x)}\ldots
    \dashint_{B_{\rho_{k}(z_{k - 1})}(z_{k - 1})} |\hat{f}(z_{k})|^{r}
    \mathrm{d} z_{k} \ldots \mathrm{d} z_{1} \bigg)^{\frac{1}{r}}.
  \end{equation*}
  For well-chosen $\rho_{1}, \ldots, \rho_{k} $, this can be made arbitrarily
  close to $M_{r}^{k}\hat{f}(x)$, so the claim holds.
\end{proof}
Hereafter, we will use the form of $M_{r, \eta, \rho, k}$ given
in~\eqref{M operator A form}. As is often the case, it will be more convenient
to work with an extension operator rather than the restriction. Given 
$g \colon \mathbb{R}^{d} \rightarrow \mathbb{C}$ and
$f \colon \mathbb{R} \rightarrow \mathbb{C}$,
\begin{align*}
  \langle M_{r, \eta, \rho, k}g, f \rangle 
  &= \int_{\mathbb{R}} \int_{\mathbb{R}^{k d}}
  \hat{g}(\gamma(t) - z_{1} - \ldots - z_{k}) 
  \mathcal{A}^{k}_{\gamma(t)}(z_{1}, \ldots, z_{k}) 
  \overline{f(t)} \mathrm{d} z_{1} \ldots \mathrm{d} z_{k} \mathrm{d} t
  \\
  &= \int_{\mathbb{R}} \int_{\mathbb{R}^{k d}} \int_{\mathbb{R}^{d}} 
  e^{-2 \pi i \xi (\gamma(t) - z_{1} - \ldots - z_{k})} 
  g(\xi) \mathcal{A}^{k}_{\gamma(t)}(z_{1}, \ldots, z_{k}) \overline{f(t)}
  \mathrm{d} \xi \mathrm{d} z_{1} \ldots \mathrm{d} z_{k} \mathrm{d} t.
\end{align*}
Hence the adjoint is given by
\begin{align*}
  M_{r, \eta, \rho, k}^{*}f(\xi)
  &= \int_{\mathbb{R}} \int_{\mathbb{R}^{k d}} 
  e^{2 \pi i \xi (\gamma(t) - z_{1} - \ldots -z_{k})} 
  \overline{\mathcal{A}^{k}_{\gamma(t)}}(z_{1}, \ldots, z_{k}) f(t)
  \mathrm{d} z_{1} \ldots \mathrm{d} z_{k} \mathrm{d} t
  \\
  &= \int_{\mathbb{R}} e^{2 \pi i \xi \gamma(t)}
  \widehat{\overline{\mathcal{A}^{k}_{\gamma(t)}}}(\xi, \ldots, \xi) f(t)
  \mathrm{d} t.
\end{align*}
Setting $\vec{\xi}^{k} = (\xi, \ldots, \xi)$, we have
\begin{equation*}
  M_{r, \eta, \rho, k}^{*}f(\xi)
  = \int_{\mathbb{R}} e^{2 \pi i \xi \gamma(t)}
  \widehat{\overline{\mathcal{A}^{k}_{\gamma(t)}}}(\vec{\xi}^{k})
  f(t) \mathrm{d} t.
\end{equation*}
Proposition~\ref{main theorem first reduction} now follows from the following
lemma:
\begin{lemma} \label{main theorem second reduction}
  Let $1 \leq p < \frac{d^{2} + d + 2 (r - 1)}{2 (r - 1)}$,
  $q = \frac{d (d + 1)}{2} p'$, and $2 \leq r < \frac{d + 2}{2}$. There is
  $C > 0$ such that for $\rho_{1}, \ldots, \rho_{k}$ and $\eta$ measurable
  satisfying~\eqref{eta assumption}, and for all Schwartz functions $f$,
  \begin{equation} \label{main theorem second reduction estimate}
    ||M_{r, \eta, \rho, k}^{*}f||_{L^{q}(\mathbb{R}^{d})} \leq 
    C ||f||_{L^{p}(\Gamma)}.
  \end{equation}
\end{lemma}
  \section{The Induction Argument} \label{proof}
  %
    % This houses the section Proof of Theorem \ref{max_restrict_thm} for the paper
% Maximal Operators and Fourier Restriction on the Moment Curve.
%
The proof of Lemma~\ref{main theorem second reduction} proceeds by induction.
The base case is  $p = 1$ and $q = \infty$. Here,
\begin{equation*}
  |M_{r, \eta, \rho, k}^{*}f(\xi)|
  = \Big| \int_{\mathbb{R}} e^{2 \pi i \xi \gamma(t)}
  \widehat{\overline{\mathcal{A}^{k}_{\gamma(t)}}}(\vec{\xi}) f(\xi)
  \mathrm{d} t \Big|
  %
  % Intermediate step where the absolute value is brought inside. If this is
  % printed, the environment should be align*.
  %
  % \\
  % &\leq \int_{\mathbb{R}}
  % |\widehat{\overline{\mathcal{A}^{k}_{\gamma(t)}}}(\vec{\xi})| |f(t)|
  % \mathrm{d} t
  %
  \leq \int_{\mathbb{R}} \sup_{x \in \mathbb{R}^{d}}
  ||\mathcal{A}^{k}_{x}||_{L^{1}(\mathbb{R}^{k d})} |f(t)| \mathrm{d} t.
\end{equation*}
By~\eqref{eta assumption},
$||\mathcal{A}^{k}_{x}||_{L^{1}(\mathbb{R}^{d})} \leq 1$ for all $x$. Thus,
\begin{equation*}
  |M_{r, \eta, \rho, k}^{*}f(\xi)|
  \leq \int_{\mathbb{R}} |f(t)| \mathrm{d} t
  = ||f||_{L^{1}(\Gamma)}.
\end{equation*}
This completes the base case. The following lemma, along with a little
arithmetic, establishes the claimed range of $p$ and $q$:
\begin{lemma} \label{induction step}
  Assume for every $1 \leq p < p_{0} < \frac{d^{2} + d + 2 (r - 1)}{2 (r - 1)}$, 
  $q = \frac{d (d + 1)}{2} p'$, there is $C > 0$ such that
  \eqref{main theorem second reduction estimate} holds for all $k$, all
  measurable
  $\rho_{1}, \dots, \rho_{k} \colon \mathbb{R}^{d} \rightarrow \mathbb{R}_{>0}$,
  all measurable
  $\eta \colon \mathbb{R}^{d} \times \mathbb{R}^{d} \rightarrow \mathbb{C}$
  satisfying~\eqref{eta assumption} for every $x \in \mathbb{R}^{d}$, and all
  $f$. Then \eqref{main theorem second reduction estimate}
  holds for all such $\eta$, $\rho$, $k$, and $f$, and for all $p$ satisfying
  \begin{equation*}
    \frac{d}{p} > \frac{2}{(d + 2) p_{0}' (r' - 1)} + \frac{d}{(d + 2) p_{0}},
  \end{equation*}
  and $q = \frac{d (d + 1)}{2} p'$.
\end{lemma}
To prove Lemma~\ref{induction step}, we adapt Drury's argument in~\cite{Drury}.
Thus, we rewrite the left-hand side
of~\eqref{main theorem second reduction estimate} as
\begin{equation*}
  ||M_{r, \eta, \rho, k}^{*}f||_{L^{q}(\mathbb{R}^{d})} = 
  ||(M_{r, \eta, \rho, k}^{*}f)^{d}||_{L^{q / d}(\mathbb{R}^{d})}^{1 / d}.
\end{equation*}
Expanding gives
\begin{equation*}
  (M_{r, \eta, \rho, k}^{*}f)^{d}(\xi) = \int_{\mathbb{R}^{d}}
  e^{2 \pi i \xi \sum_{j = 1}^{d} \gamma(x_{j})} \prod_{j = 1}^{d}
  \widehat{\overline{A_{\gamma(x_{j})}}}(\vec{\xi}^{k}) f(x_{j}) \mathrm{d} x.
\end{equation*}
Define
\begin{equation} \label{TG}
  TG(\xi) = \int_{\Gamma + \dots + \Gamma} e^{2 \pi i \xi y} 
  \prod_{j = 1}^{d}
  \widehat{\overline{\mathcal{A}^{k}_{\gamma(x_{j})}}}(\vec{\xi}^{k})
  G(y) \mathrm{d} y,
\end{equation}
where $x = x(y)$ is uniquely determined by $x_{1} < x_{2} < \dots < x_{d}$ and
$y = \sum_{j = 1}^{d} \gamma(x_{j})$. Let $v$ be the Vandermonde determinant,
\begin{equation*}
  |v(x)| = \prod_{1 \leq i < j \leq d} |x_{i} - x_{j}|.
\end{equation*}
Then with the choice
\begin{equation*}
  G(y) = \prod_{j = 1}^{d} f(x_{j}) |v(x)|^{-1},
\end{equation*}
we have
\begin{equation} \label{TG = M^d}
  TG(\xi) = \frac{1}{d!} (M_{r, \eta, \rho, k}^{*}f)^{d}(\xi).
\end{equation}
To apply the induction hypothesis, we will need to work with another change of
variables.  For $h = (h_{1}, h')$ with $h_{1} = 0 < h_{2} < \dots < h_{d}$, set
$x_{j} = t + h_{j}$ and \(\gamma_{h}(t) = \sum_{j = 1}^{d} \gamma(x_{j})\).
Define the auxiliary function
\begin{equation} \label{Tilde G}
  \Tilde{G}(t, h)
  = G(\gamma_{h}(t))
  = \prod_{j = 1}^{d} f(t + h_{j}) |v(h)|^{-1}.
\end{equation}
Finally, fix $0 < \epsilon < \frac{d + 2}{2} - r.$
\begin{lemma} \label{TG L^2}
  For $T$ defined as in~\eqref{TG} and $r + \epsilon < \frac{d + 2}{2}$, we
  have
  \begin{equation} \label{TG L^2 bound}
    ||TG||_{L^{r + \epsilon}} \leq C
    ||\Tilde{G}||_{L^{(r + \epsilon)'}_{h'}(L^{(r + \epsilon)'}_{t}; |v(h)|)}.
  \end{equation}
\end{lemma}
\begin{proof}
  For any test function $H$, by~\eqref{TG} we have
  \begin{equation*}
    |\langle TG, H \rangle|
    = \bigg| \int_{\mathbb{R}^{d}} TG(\xi) \overline{H}(\xi) \mathrm{d} \xi
    \bigg|
    = \bigg| \int_{\mathbb{R}^{d}} \int_{\Gamma + \dots + \Gamma}
    e^{2 \pi i \xi y}
    \prod_{j = 1}^{d}
    \widehat{\overline{\mathcal{A}^{k}_{\gamma(x_{j})}}}(\vec{\xi}^{k})
    G(y) \overline{H}(\xi) \mathrm{d} y \mathrm{d} \xi \bigg|.
  \end{equation*}
  %
  % Intermediate step: undo the Fourier transform on \mathcal{B}.
  %
  % \begin{equation*}
  % %
  %   |\langle TG, H \rangle|
  %   = \bigg| \int_{\mathbb{R}^{d}} G(y) \int_{\mathbb{R}^{d}}
  %   \int_{\mathbb{R}^{k d}} e^{2 \pi i \xi (y - w_{1} - \ldots - w_{k})}
  %   \overline{\mathcal{B}_{y}}(w_{1}, \ldots, w_{k}) \overline{H}(\xi)
  %   \mathrm{d} w_{1} \ldots \mathrm{d} w_{k} \mathrm{d} \xi \mathrm{d} y
  %   \bigg|.
  % %
  % \end{equation*}
  %
  Changing the order of integration and applying Plancherel in $\xi$,
  \begin{equation*}
    |\langle TG, H \rangle|
    = \bigg| \int_{\Gamma + \dots + \Gamma} G(y) \int_{\mathbb{R}^{k d}}
    \bigast_{j = 1}^{d}
    \overline{\mathcal{A}^{k}_{\gamma(x_{j})}(w_{1}, \ldots, w_{k})
    \widehat{H}(y - w_{1} - \ldots - w_{k})}
    \mathrm{d} w_{1} \ldots \mathrm{d} w_{k} \mathrm{d} y \bigg|.
  \end{equation*}
  We can now apply the following lemma, which we will prove shortly.
  \begin{lemma} \label{A convolution}
    Let $\mathcal{A}^{k}_{z}(w_{1}, \ldots, w_{k})$ be defined as
    in~\eqref{A operator}, and let $\hat{H}$ be a test function. Then for each
    $n, k \in \mathbb{N}$ and $z_{1}, \ldots, z_{n}$,
    \begin{equation} \label{A convolution bound}
      \int_{\mathbb{R}^{k d}} |(\mathcal{A}^{k}_{z_{1}} \ast \ldots \ast
      \mathcal{A}^{k}_{z_{n}})(w_{1}, \ldots, w_{k})
      \hat{H}(y - w_{1} - \ldots - w_{k})|
      \mathrm{d} w_{1} \ldots \mathrm{d} w_{k} \leq M_{r}^{n k} \hat{H}(y).
    \end{equation}
  \end{lemma}
  Using this lemma, we see that
  \begin{equation*}
    |\langle TG, H \rangle|
    \leq \int_{\Gamma + \dots + \Gamma} |G(y)| M_{r}^{k d} \hat{H}(y)
    \mathrm{d} y.
  \end{equation*}
  By H\"older's inequality, we obtain
  \begin{equation*}
    |\langle TG, H \rangle|
    \leq ||G||_{(r + \epsilon)'} ||M_{r}^{k d}\hat{H}||_{r + \epsilon}.
  \end{equation*}
  Since $r < r + \epsilon$, we have
  \begin{equation*}
    |\langle TG, H \rangle|
    \leq ||G||_{(r + \epsilon)'} ||\hat{H}||_{r + \epsilon}.
  \end{equation*}
  Finally, $r + \epsilon > 2$, so by Hausdorff-Young,
  \begin{equation*}
    |\langle TG, H \rangle|
    \leq ||G||_{(r + \epsilon)'} ||H||_{(r + \epsilon)'}.
  \end{equation*}
  Thus, TG is bounded from $L^{(r + \epsilon)'}$ to $L^{r + \epsilon}$, which,
  along with a change of variables, proves the lemma.
\end{proof}
Now we prove Lemma~\ref{A convolution}.
\begin{proof}
  Fix $k \geq 1$. We proceed by induction. The base case is $n = 1$. In this
  case, the left-hand side of~\eqref{A convolution bound} is
  \begin{align*}
    \int_{\mathbb{R}^{k d}} &|\hat{H}(y - w_{1} - \ldots - w_{k})
    \eta(z, z - w_{1} - \ldots - w_{k})
    \\
    &\cdot \chi_{\rho_{1}(z)}(w_{1}) \ldots
    \chi_{\rho_{k}(z - w_{1} - \ldots - w_{k - 1})}(w_{k})|
    \mathrm{d} w_{1} \ldots \mathrm{d} w_{k}.
  \end{align*}
  Applying H\"older's inequality, this is bounded by
  \begin{align*}
    &\bigg(\int_{\mathbb{R}^{k d}} |\hat{H}(y - w_{1} - \ldots - w_{k})|^{r}
    \big|\chi_{\rho_{1}(z)}(w_{1})\big| \ldots
    \big|\chi_{\rho_{k}(z - w_{1} - \ldots - w_{k - 1})}(w_{k})\big|
    \mathrm{d} w_{1} \ldots \mathrm{d} w_{k}\bigg)^{\frac{1}{r}}
    \\
    &\cdot \bigg(\int_{\mathbb{R}^{k d}}
    |\eta(z, z - w_{1} - \ldots - w_{k})|^{r'}
    \big|\chi_{\rho_{1}(z)}(w_{1})\big| \ldots
    \big|\chi_{\rho_{k}(z - w_{1} - \ldots - w_{k - 1})}(w_{k})\big|
    \mathrm{d} w_{1} \ldots \mathrm{d} w_{k}\bigg)^{\frac{1}{r'}}.
  \end{align*}
  Changing variables in each integral transforms the above into
  \begin{align*}
    &\bigg(\ \dashint_{B_{\rho_{1}(z)}(y)} \ldots 
    \dashint_{B_{\rho_{k}(v_{k - 1} + z - y)}(v_{k - 1})} |\hat{H}(v_{k})|^{r} 
    \mathrm{d} v_{k} \ldots \mathrm{d} v_{1}\bigg)^{\frac{1}{r}}
    \\
    &\cdot \bigg(\ \dashint_{B_{\rho_{1}(z)}(z)} \ldots 
    \dashint_{B_{\rho_{k}(z - v_{k - 1})}(v_{k - 1})} |\eta(z, v_{k})|^{r'}
    \mathrm{d} v_{k} \ldots \mathrm{d} v_{1}\bigg)^{\frac{1}{r'}}
  \end{align*}
  By~\eqref{eta assumption}, the second term is bounded by 1. Moreover, the
  first term is bounded by $M_{r}^{k} \hat{H}(y)$, so the base case is done.
  Now, assume we have \eqref{A convolution bound} for some $n$ and all
  functions $\hat{H}$. We want to bound
  \begin{equation} \label{A convolution n + 1}
    \int_{\mathbb{R}^{k d}} |(\mathcal{A}^{k}_{z_{1}} \ast \ldots \ast 
    \mathcal{A}^{k}_{z_{n + 1}})(w_{1}, \ldots, w_{k}) 
    \hat{H}(y - w_{1} - \ldots - w_{k})|
    \mathrm{d} w_{1} \ldots \mathrm{d} w_{k},
  \end{equation}
  with the convolution performed $n + 1$ times. Split up the convolution as the
  convolution of an $n$-fold convolution with $A_{z_{n + 1}}$ to
  rewrite~\eqref{A convolution n + 1} as
  \begin{equation*}
    \int_{\mathbb{R}^{k d}} |(\mathcal{A}^{k}_{z_{1}} \ast \ldots
    \ast \mathcal{A}^{k}_{z_{n}}) \ast
    \mathcal{A}^{k}_{z_{n + 1}}(w_{1}, \ldots, w_{k})
    \hat{H}(y - w_{1} - \ldots - w_{k})|
    \mathrm{d} w_{1} \ldots \mathrm{d} w_{k}.
  \end{equation*}
  Expanding this convolution, we can further
  rewrite~\eqref{A convolution n + 1} as
  \begin{align*}
    \int_{\mathbb{R}^{k d}} \int_{\mathbb{R}^{k d}}
    &|(\mathcal{A}^{k}_{z_{1}} \ast \ldots \ast \mathcal{A}^{k}_{z_{n}})
    (v_{1}, \ldots, v_{k}) 
    \mathcal{A}^{k}_{z_{n + 1}}(w_{1} - v_{1}, \ldots, w_{k} - v_{k})
    \\
    &\cdot \hat{H}(y - w_{1} - \ldots - w_{k})| \mathrm{d} v_{1} \ldots
    \mathrm{d} v_{k} \mathrm{d} w_{1} \ldots \mathrm{d} w_{k}.
  \end{align*}
  With the change of variables
  $u_{j} = w_{j} - v_{j}$,~\eqref{A convolution n + 1} becomes
  \begin{align*}
    \int_{\mathbb{R}^{k d}} \int_{\mathbb{R}^{k d}}
    &|(\mathcal{A}^{k}_{z_{1}} \ast \ldots \ast \mathcal{A}^{k}_{z_{n}})
    (v_{1}, \ldots, v_{k}) \mathcal{A}^{k}_{z_{n + 1}}(u_{1}, \ldots, u_{k})
    \\
    &\cdot \hat{H}(y - v_{1} - \ldots - v_{k} - u_{1} - \ldots - u_{k})|
    \mathrm{d} v_{1} \ldots \mathrm{d} v_{k} \mathrm{d} u_{1} \ldots
    \mathrm{d} u_{k}.
  \end{align*}
  By the induction hypothesis, the above is bounded by
  \begin{equation*}
    \int_{\mathbb{R}^{k d}} |\mathcal{A}^{k}_{z_{n + 1}}(v_{1}, \ldots, v_{k})
    M_{r}^{n k}\hat{H}(y - v_{1} - \ldots - v_{k})|
    \mathrm{d} v_{1} \ldots \mathrm{d} v_{k}.
  \end{equation*}
  Finally, another application of the induction hypothesis shows
  that~\eqref{A convolution n + 1} is bounded by $M_{r}^{(n + 1) k}\hat{H}(y)$.
\end{proof}
\begin{lemma}
  For $T$ defined as in~\eqref{TG}, there is a constant $C_{p, r}$ such that
  \begin{equation} \label{TG L^p bound}
    ||TG||_{L^{q}} \leq C_{p, r} ||\Tilde{G}||_{L^{1}_{h'}(L^{p}_{t}; |v(h)|)}.
  \end{equation}
\end{lemma}
\begin{proof}
  By Minkowski's inequality for integrals,
  \begin{align*}
    ||TG||_{L^{q}}
    &= \bigg| \bigg| \int_{0}^{\infty} \int_{h_{2}}^{\infty} \dots
    \int_{h_{d - 1}}^{\infty} \int_{\mathbb{R}}
    e^{2 \pi i \xi \gamma_{h}(t)}
    \prod_{j = 1}^{d}
    \widehat{\overline{\mathcal{A}^{k}_{\gamma(t + h_{j})}}}(\vec{\xi}^{k})
    \Tilde{G}(t, h) v(h) \mathrm{d} t \mathrm{d} h' \bigg| \bigg|_{L^{q}_{\xi}}
    \\
    &\leq \bigg| \bigg| \int_{\mathbb{R}} e^{2 \pi i \xi \gamma_{h}(t)}
    \prod_{j = 1}^{d}
    \widehat{\overline{\mathcal{A}^{k}_{\gamma(t + h_{j})}}} (\vec{\xi}^{k})
    \Tilde{G}(t, h) v(h) \mathrm{d} t \bigg| \bigg|_{L^{1}_{h} L^{q}_{\xi}}.
  \end{align*}
  Define the operator
  \begin{equation*}
    S_{h}F(\xi) = \int_{\mathbb{R}} e^{2 \pi i \xi \gamma_{h}(t)}
    \prod_{j = 1}^{d}
    \widehat{\overline{\mathcal{A}^{k}_{\gamma(t + h_{j})}}}(\vec{\xi}^{k}) F(t)
    \mathrm{d} t,
  \end{equation*}
  whose adjoint is given by
  \begin{align*}
    S_{h}^{*} H(t)
    &= \int_{\mathbb{R}^{d}} e^{-2 \pi i \xi \gamma_{h}(t)}
    \prod_{j = 1}^{d}
    \widecheck{\mathcal{A}^{k}_{\gamma_(t + h_{j})}}(\vec{\xi}^{k}) H(\xi)
    \mathrm{d} \xi
    \\
    &= \int_{\mathbb{R}^{k d}} \bigast_{j = 1}^{d}
    \mathcal{A}^{k}_{\gamma(t + h_{j})}(w_{1}, \ldots, w_{k})
    \hat{H}(\gamma_{h}(t) - w_{1} - \ldots - w_{k})
    \mathrm{d} w_{1} \ldots \mathrm{d} w_{k}.
  \end{align*}
  Using Lemma~\ref{A convolution}, we obtain the bound
  \begin{equation*}
    |S_{h}^{*} H(t)|
    \leq M_{r}^{k} \hat{H}(\gamma_{h}(t)).
  \end{equation*}
  Since each $\gamma_{h}$ is an affine transformation of the original curve,
  the induction hypothesis yields
  \begin{equation*}
    ||S_{h}^{*}H||_{L^{p'}} \leq C_{p, r} ||H||_{L^{q'}}.
  \end{equation*}
  Hence, we have
  \begin{equation*}
    ||S_{h}F||_{L^{q}_{\xi}} \leq C_{p, r} ||F||_{L^{p}}.
  \end{equation*}
  Setting $F(t) = \Tilde{G}(t, h) v(h)$ for each $h$ and integrating in $h'$
  finishes the proof.
\end{proof}
By interpolating \eqref{TG L^2 bound} and \eqref{TG L^p bound}, we obtain
\begin{equation} \label{TG interpolation bound}
  ||TG||_{L^{c}} \leq C_{a, b, r} ||\Tilde{G}||_{L^{a}_{h'}(L^{b}_{t}; |v(h)|)}
\end{equation}
for all $(a^{-1}, b^{-1})$ in the triangle with vertices $(1, 1)$,
$(1, p_{0}^{-1})$, and $(((r + \epsilon)')^{-1}, ((r + \epsilon)')^{-1})$,
with $c$ satisfying
\begin{equation*}
  \frac{(d + 2) (d - 1)}{2} a^{-1} + b^{-1} + 
  \frac{d (d + 1)}{2} c^{-1} = \frac{d (d + 1)}{2}.
\end{equation*}
Expanding out $\Tilde{G}$ using \eqref{Tilde G}, we see that
\begin{equation*}
  ||\Tilde{G}||_{L^{a}_{h'}(L^{b}_{t}; |v(h)|)} = \bigg(\int_{\mathbb{R}} 
  |v(h)|^{-(a - 1)} \bigg(\int_{\mathbb{R}^{d - 1}} 
  |f(t + h_{1}) \ldots f(t + h_{d})|^{b} \mathrm{d} t\bigg)^{\frac{a}{b}}
  \mathrm{d} h'\bigg)^{\frac{1}{a}}.
\end{equation*}
\newpage
As noted in \cite{Drury}, $v(0, h')^{-1} \in L^{\frac{d}{2}, \infty}_{h'}$, so
we can apply H\"older's inequality to obtain
\begin{equation} \label{G L^aL^b Holder}
  ||\Tilde{G}||_{L^{a}_{h'}(L^{b}_{t}; |v(h)|)} \leq ||f||_{L^{p, 1}_{t}}^{d}
\end{equation}
for
\begin{equation*}
  \begin{cases}
    1 < a < \frac{d + 2}{2},
    \\
    a \leq b < \frac{2 a}{d + 2 - d a}, \text{ and}
    \\
    \frac{d}{p}
    = \frac{(d + 2) (d - 1)}{2} a^{-1} + b^{-1} - \frac{d (d - 1)}{2}.
  \end{cases}
\end{equation*}
%
% \begin{equation*}
% %
%   1 < a < \frac{d + 2}{2},
%   \quad a \leq b < \frac{2 a}{d + 2 - d a},
%   \quad \frac{d}{p} = \frac{(d + 2) (d - 1)}{2} a^{-1}
%   + b^{-1} - \frac{d (d - 1)}{2}.
% %
% \end{equation*}
%
Plugging \eqref{TG = M^d} and \eqref{G L^aL^b Holder} into
\eqref{TG interpolation bound},
\begin{equation} \label{M operator L^q to L^p,1}
||M_{r, \eta, \rho, k}f||_{L^{q}} \lesssim ||f||_{L^{p, 1}},
\end{equation}
for
\begin{equation} \label{p of a and b}
  \frac{d}{p}
  = \frac{(d + 2) (d - 1)}{2} a^{-1} + b^{-1} - \frac{d (d - 1)}{2},
\end{equation}
where $q = \frac{d (d + 1)}{2} p'$, and $a$ and $b$
satisfy~(Figure~\ref{fig: interpolation picture}):
\begin{equation*}
  \begin{cases}
    \frac{d}{d + 2} < a^{-1} < 1,
    \\
    b^{-1} \leq a^{-1},
    \\
    (d+ 2)a^{-1} - 2 b^{-1} < d, \text{ and}
    \\
    (p_{0} - (r + \epsilon)') a^{-1} + p_{0} ((r + \epsilon)' - 1)
    b^{-1} \geq p_{0} - 1.
  \end{cases}
\end{equation*}
%
% \begin{equation*}
% %
%   \frac{d}{d + 2} < a^{-1} < 1, 
%   \quad b^{-1} \leq a^{-1}, 
%   \quad (d+ 2)a^{-1} - 2 b^{-1} < d,
% %
% \end{equation*}
% %
% \begin{equation*}
% %
%   (p_{0} - (r + \epsilon)') a^{-1} + p_{0} ((r + \epsilon)' - 1)
%   b^{-1} \geq p_{0} - 1.
% %
% \end{equation*}
%
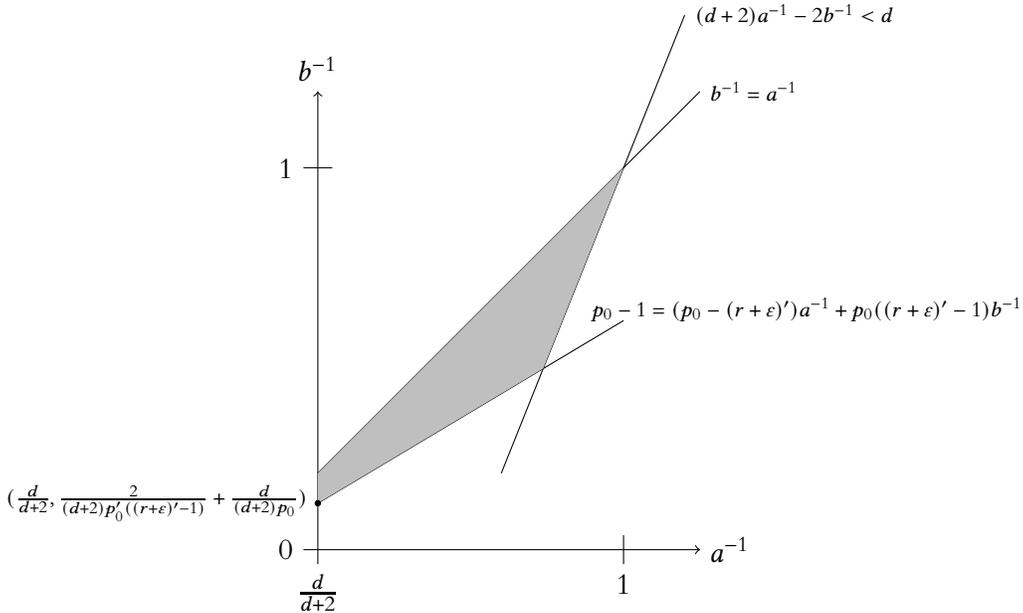
\begin{figure}[H]
  \centering
  \begin{tikzpicture}[x = 4 in, y = 4 in]
  %
    % Most nodes are \scriptsize to improve readability. Looked for a few
    % minutes for a global solution, but didn't find one that worked.
    %
    % y axis.
    %
    \draw [name path = vertical, ->] ($(.6, .5) - (0, .5 em)$) node [below]
    {$\frac{d}{d + 2}$}
    -- (.6, 1.1) node [above] {$b^{-1}$};
    %
    % x axis.
    %
    \draw [->] ($(.6, .5) - (.5 em, 0)$) node [left] {0}
    -- (1.1, .5) node [right] {$a^{-1}$};
    %
    % Mark off 1 on each axis.
    %
    \draw ($(1, .5) - (0, .5 em)$) node [below] {1}
    -- ($(1, .5) + (0, .5 em)$);
    \draw ($(.6, 1) - (.5 em, 0)$) node [left] {1}
    -- ($(.6, 1) + (.5 em, 0)$);
    %
    % The line y = x.
    %
    \draw [name path = equalslope] (.6, .6)
    -- (1.1, 1.1) node [right] {\scriptsize $b^{-1} = a^{-1}$};
    %
    % The line with the steepest slope.
    %
    \draw [name path = steepslope] (.84, .6) -- (1.08, 1.2) node [right]
    {\scriptsize $(d + 2)a^{-1} - 2 b^{-1} < d$};
    %
    % The lowest line with positive slope.
    %
    \draw [name path = shallowslope] (.6, .5608)
    -- (1, .8);
    \node at (1.24, .815)
    {\scriptsize $p_{0} - 1
    = (p_{0} - (r + \epsilon)') a^{-1} + p_{0} ((r + \epsilon)' - 1) b^{-1}$};
    %
    % Find intersections of the above lines.
    %
    \path [name intersections = {of = vertical and equalslope}];
    \coordinate (A) at (intersection-1);
    \path [name intersections = {of = equalslope and steepslope}];
    \coordinate (B) at (intersection-1);
    \path [name intersections = {of = steepslope and shallowslope}];
    \coordinate (C) at (intersection-1);
    \path [name intersections = {of = shallowslope and vertical}];
    \coordinate (D) at (intersection-1);
    %
    % Fill the region bounded by the above lines.
    %
    \fill [fill = lightgray] (A) -- (B) -- (C) -- (D) -- cycle;
    %
    % Label our special point.
    %
    \coordinate [label = left: {\scriptsize $(\frac{d}{d + 2},
    \frac{2}{(d + 2) p_{0}' ((r + \epsilon)' - 1)} + \frac{d}{(d + 2) p_{0}})$}]
    (thepoint) at (.6, .5608);
    \filldraw (thepoint) circle [radius = 1 pt];
  \end{tikzpicture}
  \caption{\small \textit{Range of $a$ and $b$ for
  which~\eqref{M operator L^q to L^p,1} holds with $p$
  satisfying~\eqref{p of a and b} and $q = \frac{d (d + 1)}{2} p'$.}}
  \label{fig: interpolation picture}
\end{figure}
Since $r + \epsilon \leq \frac{d + 2}{2}$, the point
$(a^{-1}, b^{-1}) = (\frac{d}{d + 2}, 
\frac{2}{(d + 2) p_{0}' ((r + \epsilon)' - 1)} + \frac{d}{(d + 2) p_{0}})$
lies on the boundary of this region and satisfies
\begin{equation*}
  \frac{(d + 2) (d - 1)}{2} a^{-1} + b^{-1} - \frac{d (d - 1)}{2}
  < \frac{d}{p_{0}}.
\end{equation*}
Taking $(a^{-1}, b^{-1})$ slightly inside of the region and using real
interpolation, we obtain
\begin{equation*}
  ||M_{r, \eta, \rho, k}f||_{L^{q}} \lesssim ||f||_{L^{p}}, 
  \quad q = \frac{d (d + 1)}{2} p', 
  \quad \frac{d}{p}
  > \frac{2}{(d + 2) p_{0}' ((r + \epsilon)' - 1)} + \frac{d}{(d + 2) p_{0}}.
\end{equation*}
Since this is true for all $0 < \epsilon < \frac{d + 2}{d} - r$, we have
\begin{equation*}
  ||M_{r, \eta, \rho, k}f||_{L^{q}} \lesssim ||f||_{L^{p}}, 
  \quad q = \frac{d (d + 1)}{2} p', 
  \quad \frac{d}{p}
  > \frac{2}{(d + 2) p_{0}' (r' - 1)} + \frac{d}{(d + 2) p_{0}},
\end{equation*}
which proves Lemma~\ref{main theorem second reduction} and hence
Theorem~\ref{main theorem}.

%%% Local Variables:
%%% mode: latex
%%% TeX-master: "main"
%%% End:

  %
  \section{Bounds on $r$} \label{counterexample}
  %
    % The counterexample to restrict the value of r to r < p'
%
%
We are not able to show, nor do we believe, that the range of $r$ is sharp.
The following proposition shows that $r \leq p'$ is necessary in any bound of
the form \eqref{main theorem estimate}, which corresponds to
$r \leq \frac{d^{2} + d + 2}{2}$ in the full Drury range. This counterexample in
dimension $d = 2$ is due to Ramos~\cite{Ramos2}.
\begin{proposition} \label{CounterexampleOnRProp}
  Suppose that for some $p$, $q$, and $r$, and all $f \in L^{p}(\mathbb{R}^{d})$,
  we have the bound
  \begin{equation} \label{main theorem estimate again}
    ||M_{r}\hat{f}||_{L^{q}(\Gamma)}
    \leq C_{p, r} ||f||_{L^{p}(\mathbb{R}^{d})}.
  \end{equation}
  Then $r \leq p'$.
\end{proposition}
\begin{proof}
  For $0 < t < 1$, let $\hat{f_{t}} = \chi_{[-t, t]^{d}}$, and let
  $k = 1$. We first compute $f_{t}$.
  \begin{align*}
    f_{t}(x)
    &= \int_{\mathbb{R}^{d}} e^{2 \pi i x \xi} \chi_{[-t, t]}(\xi) \mathrm{d} \xi
    = \int_{-t}^{t} \dots \int_{-t}^{t} e^{2 \pi i x \xi} \mathrm{d} \xi
    \\
    &= \prod_{j = 1}^{d}
    \frac{e^{2 \pi i x_{j} t} - e^{-2 \pi i x_{j} t}}{2 \pi i x_{j}}
    = \prod_{j = 1}^{d} \frac{\sin(2 \pi x_{j} t)}{\pi x_{j}}.
  \end{align*}
  Thus, we have
  \begin{equation*}
    ||f_{t}||_{L^{p}(\mathbb{R}^{d})}
    = \bigg( \int_{\mathbb{R}^{d}} \bigg| \prod_{j = 1}^{d}
    \frac{\sin(2 \pi x_{j} t)}{\pi x_{j}} \bigg|^{p} \mathrm{d} x
    \bigg)^{\frac{1}{p}}
    = \bigg( \int_{\mathbb{R}^{d}} \prod_{j = 1}^{d} \bigg( \bigg|
    \frac{2 t \sin(y_{j})}{y_{j}} \bigg|^{p} \cdot \frac{1}{2 \pi t} \bigg)
    \mathrm{d} x \bigg)^{\frac{1}{p}}
    = C t^{\frac{d}{p'}}.
  \end{equation*}
  For $x \in [-1, 1]^{d}$, by taking the ball centered at $x$ with radius 10, we
  see that
  \begin{equation*}
    M_{r}\hat{f_{t}}(x) \gtrsim t^{\frac{d}{r}}.
  \end{equation*}
  Hence, we have
  \begin{equation*}
    ||M_{r}\hat{f_{t}}||_{L^{q}(\Gamma)} \gtrsim t^{\frac{d}{r}}.
  \end{equation*}
  Combining these estimates and \eqref{main theorem estimate again} gives
  \begin{equation*}
    t^{\frac{d}{r}} \lesssim t^{\frac{d}{p'}}.
  \end{equation*}
  Sending $t \rightarrow 0$ shows that
  \begin{equation*}
    \frac{d}{r} \geq \frac{d}{p'},
  \end{equation*}
  which means that $r \leq p'$.
\end{proof}
  \newpage
  \printbibliography
\end{document}